\documentclass[12pt]{article}
\usepackage{amsfonts}
\usepackage{amsmath,amscd, amsthm, tabularx}
\newtheorem{prop}{Proposition}[section]

\newtheorem{lemma}[prop]{Lemma}
\theoremstyle{remark}

\theoremstyle{definition}

\newtheorem{definition}[prop]{Definition}
\usepackage{graphicx}
\usepackage{cite}
\usepackage{pbox, url}

\title{Resolution of degenerate mirror families via toric morphisms}
\author{Karl Fredrickson}
\date{}
\begin{document}
\maketitle
\begin{abstract}  This paper continues the study of two examples of extremal transitions between families of Calabi-Yau threefolds.  In a previous paper we suggested that the ``mirror transition'' between mirror families predicted by Morrison could be achieved naturally by combining a toric morphism with the Batyrev-Borisov construction.  This was carried out for a particular example of a conifold transition. In this paper we show that similar methods work for another extremal transition involving more complicated singularities.  We also study how the resolution is related to geometry of the ambient toric varieties, and discuss the connection with recent work by Doran and Harder. \end{abstract}
\section{Introduction}
This paper is a continuation of a previous paper \cite{fred} which explored the general idea of studying transitions between Calabi-Yau threefolds by using toric morphisms and the Batyrev-Borisov construction.  The definition of a ``transition'' between two nonsingular Calabi-Yau threefolds $X$ and $Y$ (which we will always take to be projective varieties over $\mathbb{C}$) involves degenerating $X$ to a singular variety $X_0$, then obtaining $Y$ as a resolution of singularities of $X_0$.  One of the better-studied types of transitions is the class of {\it conifold transitions}, where $X_0$ is a variety with a finite number of ordinary double points (also called nodes) as singularities.  However, transitions where $X_0$ has singularities other than just isolated nodes can also be considered.  These are sometimes referred to as ``extremal transitions'' or ``geometric transitions''; for further background and complete definitions, see \cite{rossi, morrison}.  

The fundamental idea connecting transitions and mirror symmetry, first introduced in \cite{morrison}, is that if two Calabi-Yau manifolds are related by a transition, then their mirrors also should be.  One reason this idea is important is that it can be used to construct mirrors of threefolds for which no other mirror constructions are currently available (see \cite{BCKS, batkreuz}).

It is natural to ask about the relationship between Calabi-Yau transitions and the Batyrev-Borisov construction, which is the standard mirror construction for complete intersection Calabi-Yau manifolds in toric varieties.  The basic philosophy from \cite{fred} is that if a transition is induced by a toric morphism, which is just a linear map which behaves well with respect to fans, then the mirror transition should be induced by the dual linear map.  In \cite{fred} we showed that this works at the level of birational morphisms for two specific examples, and for one example showed that the morphism could be extended to a complete toric variety containing the entire family of smooth compactified CY varieties.  The main purpose of this paper is to carry out extension of the toric morphism for the other example, and also study how the toric morphism acts as a resolution of singularities of the degenerate mirror family.

The behavior of toric Calabi-Yau families under toric morphisms has also been studied in the papers \cite{k3} and \cite{toricmorph}.  The general theme is to study fibrations of Calabi-Yau varieties that can be realized as a toric morphism from the ambient toric variety onto a lower-dimensional toric variety.  Then all the methods of toric geometry can be used to study the fibration.  In our approach, the idea of using toric morphisms between the ambient toric varieties is similar, although the fact that we are dealing with birational morphisms of families, rather than fibrations, makes the behavior at the level of the Calabi-Yau varieties quite different. 

Recently, Doran and Harder in \cite{dh} described a general method for producing a mirror birational morphism from a toric degeneration.  The birational morphisms from \cite{fred} are actually specific cases of this construction.  We will explain the details in Section 4.

The organization of the paper is as follows.  In Section 2 we fix notation and recall relevant details from \cite{fred}.  In Section 3 we show that for another degenerate mirror family, it is possible to construct a toric morphism that acts as a resolution of singularities, similar to the main example from \cite{fred}.  Then we calculate the singular locus of the degenerate family and study how the resolution is related to toric geometry of the ambient toric varieties.  In Section 4 we explain connections to Doran and Harder's construction in \cite{dh}.

\section{Notation and geometric setup}
If $V$ is a real vector space and $S_1$, \dots, $S_n$ are subsets of $V$, then $$Conv(S_1, \dots, S_n)$$  is the convex hull of $S_1 \cup \dots \cup S_n$.  By ``cone over'' a subset $S \subseteq V$ we will always mean the set \[ \mathbb{R}_{\geq 0} S = \{ r s \ | \ r \in \mathbb{R}, r \geq 0, s \in S \}. \]

If $P \subseteq V$ is a compact convex polytope with the origin in its interior, then $\Sigma(P)$ will denote the complete fan consisting of cones over proper faces of $P$.  If $\Sigma$ is any fan then $X(\Sigma)$ will denote the toric variety associated to $\Sigma$.  We may sometimes also use $X(P)$ for the toric variety associated to $\Sigma(P)$.  The dual polytope $P^*$ is contained in the dual space $V^*$ and defined as $$P^* = \{ v \in V^* \ | \ \langle v, p \rangle \geq -1, \forall p \in P \},$$ where $\langle , \rangle$ is the dual pairing between $V$ and $V^*$.

In this paper, all piecewise linear functions on a real vector space, such as $\varphi : V \rightarrow \mathbb{R}$, will be lower convex, meaning that for any $u, v \in V$ and $0 \leq a, b \leq 1$ with $a+b = 1$, we have that $$\varphi(au+bv) \leq a \varphi(u) + b \varphi (v).$$
Given such a function we can define its Newton polytope
$$Newt(\varphi) = \{ u \in V^* \ | \ \langle u, v \rangle \geq -\varphi(v), \ \forall v \in V \}.$$ 
With our conventions, a strictly lower convex integral piecewise linear function $\varphi$ on a complete fan $\Sigma$ corresponds to an ample line bundle $\mathcal{L}$ on $X(\Sigma)$, and lattice points in $Newt(\varphi)$ correspond to monomial global sections of $\mathcal{L}$.

In \cite{fred} we constructed a birational morphism between two families of Calabi-Yau varieties, one of which was a singular family where the generic member had a singular locus consisting of four ordinary double points (nodes), and one of which had generically nonsingular members.  The families came from applying the Batyrev-Borisov construction (defined in \cite{BB}) and thus were embedded as complete intersections in toric varieties, and the morphism between families was induced by a toric morphism between ambient toric varieties.  Both of the families were (partial) resolutions of families in singular Gorenstein toric Fano varieties, arising from the fan of cones over a reflexive polytope.  The resolutions were obtained by a so-called MPCP resolution (as defined in \cite{Ba}) of the toric varieties corresponding to a maximal lattice subdivision of the reflexive polytopes.  

Let us review the relevant details from \cite{fred}.  For the rest of the paper, we fix $M' = \mathbb{Z}^5$, $M = \mathbb{Z}^4$, $M_\mathbb{R} = M \otimes \mathbb{R}$, and $M'_\mathbb{R} = M' \otimes \mathbb{R}$.  Also define the dual spaces $N = Hom(M, \mathbb{Z})$, $N' = Hom(M', \mathbb{Z})$, $N_\mathbb{R} = N \otimes \mathbb{R}$, and $N'_\mathbb{R} = N' \otimes \mathbb{R}$.

We have a smooth family $\mathcal{X}^*_{BB}$ which arises from applying the Batyrev-Borisov construction to the family of $(2,4)$ complete intersections in $\mathbb{P}^5$, and a family $\mathcal{X}^*_C$ which is a degenerate (singular) subfamily of the mirror to quartic hypersurfaces in the toric variety $P(2,4) \subseteq \mathbb{P}^5$.  (The equation for $P(2,4)$ is $z_2 z_3 = z_4 z_5$, where $z_0, \dots, z_5$ are homogeneous coordinates on $\mathbb{P}^5$.  Also, since there are many different possible choices of MPCP resolution of the ambient toric variety, the mirror family $\mathcal{X}^*_{BB}$ is not unique.  We will abuse notation by using $\mathcal{X}^*_{BB}$ to refer to families with different choices of resolution in different parts of the paper.)

The fact that a small resolution of $\mathcal{X}^*_C$ should yield the family $\mathcal{X}^*_{BB}$ was previously discussed in \cite{BCKS} (section 2.1), in the context of mirror symmetry for Calabi-Yau complete intersections in Grassmannians.  

The family $\mathcal{X}^*_{BB}$ lies in a MPCP resolution of a toric variety $X(\nabla)$ associated to the fan $\Sigma(\nabla)$ of cones over faces of a reflexive polytope $\nabla \subseteq N'_\mathbb{R}$.  Since $M'_\mathbb{R} \cong \mathbb{R}^5$ and $N'_\mathbb{R}$ is the dual space to $M'_\mathbb{R}$, we will use as a basis for $N'_\mathbb{R}$ the dual basis to the standard basis of $\mathbb{R}^5$.  Then the polytope $\nabla$ is equal to the convex hull $Conv(\nabla_1, \nabla_2)$, where $\nabla_1$ is the convex hull of the rows of the matrix 
\[ \left( \begin{array}{ccccc}
-1 & -1 & -1 & -1 & 0\\
3 & -1 & -1 & -1 & 0 \\
-1 & 3 & -1 & -1 & 0 \\
-1 & -1 & 3 & -1 & 0 \\ 
-1 & -1 & -1 & 3 & 0 \\
-1 & -1 & -1 & -1 & 4
\end{array} \right)\]
and $\nabla_2$ is the convex hull of the rows of 
\[ \left( \begin{array}{ccccc}
0 & 0 & 0 & 0 & -1 \\
2 & 0 & 0 & 0 & -1 \\
0 & 2 & 0 & 0 & -1 \\
0 & 0 & 2 & 0 & -1 \\ 
0 & 0 & 0 & 2 & -1 \\
0 & 0 & 0 & 0 & 1 
\end{array} \right).\]
The vertices of $\nabla$ consist of all rows of the above two matrices.  The family $\mathcal{X}^*_{BB}$ is defined on the open torus Spec $\mathbb{C}[M'] \subseteq X(\nabla)$ by \begin{align} \label{bbeqns}
-1 + Y_1 + Y_2 + Y_3 + b_4 Y_4 &= 0 \\
\label{bbeqns2} -1+Y_5 + (Y_1Y_2Y_3Y_4Y_5)^{-1} &= 0
\end{align}
where $Y_i = z^{e_i}$, and $e_1, \dots, e_5 \in M'_\mathbb{R}$ is the standard basis, and $b_4 \in \mathbb{C}$ is a generic coefficient.  The LHS of each equation may be regarded as a global section of a line bundle on $X(\nabla)$.  We define the piecewise linear functions $\varphi^i: N'_\mathbb{R} \rightarrow \mathbb{R}$ on $\Sigma(\nabla)$ for $i = 1,2$ by $\varphi^i = 1$ on all vertices of $\nabla_i$ and $\varphi^i = 0$ on all other vertices of $\nabla$.  Then the LHS of the first equation may be regarded as a global section of the line bundle $\mathcal{L}_1$ associated to $\varphi^1$, and the LHS of the second equation may be regarded as a global section of $\mathcal{L}_2$ associated to $\varphi^2$.  Taking the zero locus of these global sections will define a singular family in $X(\nabla)$.  

The other family under consideration, $\mathcal{X}^*_C$, is defined as follows. The polytope $\Delta^*_{P(2,4)} \subseteq N_\mathbb{R}$ is the convex hull of six vertices which are the rows of the matrix 
\[ \left( \begin{array}{cccc}
-1 & -1 & -1 & -1 \\
3 & -1 & -1 & -1 \\
-1 & 3 & -1 & -1 \\
-1 & -1 & 3 & -1 \\
3 & -1 & -1 & 3 \\
-1 & 3 & -1 & 3
\end{array} \right)\]
where we use the dual basis to $M_\mathbb{R} \cong \mathbb{R}^4$.  (This is the Newton polytope of quartics on $P(2,4)$.)  Let $X(\Delta^*_{P(2,4)})$ be the toric variety associated to the fan $\Sigma(\Delta^*_{P(2,4)})$ of cones over the faces of $\Delta^*_{P(2,4)}$.  The anticanonical bundle on $X(\Delta^*_{P(2,4)})$ has a basis of global sections consisting of lattice points contained in the dual polytope, $\Delta_{P(2,4)} \subseteq M_\mathbb{R}$, and this polytope is the convex hull of the rows of the matrix 
\[ \left( \begin{array}{cccc}
1 & 0 & 0 & 0 \\
0 & 1 & 0 & 0 \\
0 & 0 & 1 & 0 \\
0 & 0 & 0 & 1 \\
-1 & -1 & -1 & 0 \\
1 & 1 & 0 & -1
\end{array} \right).\]
Thus, the equation 
\begin{equation*} \label{cfoldfamily}
-1+X_1+X_2+X_3+X_4+a_5(X_1X_2X_3)^{-1}+a_6X_1X_2X^{-1}_4= 0
\end{equation*}
with $X_i = z^{e_i}$ where $e_1, \dots, e_4 \in M_\mathbb{R}$ is the standard basis and $a_5$ and $a_6$ are generic coefficients, defines a family of singular Calabi-Yau hypersurfaces in $X(\Delta^*_{P(2,4)})$.  After an MPCP resolution of $X(\Delta^*_{P(2,4)})$ to a new toric variety $\widehat{X} (\Delta^*_{P(2,4)})$ the family is generically nonsingular.  We define $\mathcal{X}^*_C \subseteq \widehat{X} (\Delta^*_{P(2,4)})$ as the degenerate subfamily with $a_6 = 1$.  Generic members of this family have a singular locus with four ordinary double points (see \cite{BCKS}, section 2.1).

In \cite{fred} we showed that with properly chosen MPCP resolutions, there is a toric morphism from $\widehat{X}(\nabla)$ to $\widehat{X}(\Delta^*_{P(2,4)})$.  This morphism is given by an integral linear map $g^* : N'_\mathbb{R} \rightarrow N_\mathbb{R}$ which satisfies the nice property that $g^*(\nabla) = \Delta^*_{P(2,4)}$.  Furthermore, after setting the complex structure parameters $b_4$ and $a_5$ equal, this toric morphism induces a regular birational map between the families $\mathcal{X}^*_{BB}$ and $\mathcal{X}^*_C$.

The results from \cite{fred} established that $g^*$ is a resolution of singularities, meaning that for a generic member $Z_C \in \mathcal{X}^*_C$, and its corresponding member $Z_{BB} \in \mathcal{X}^*_{BB}$, the map $g^* : Z_{BB} \rightarrow Z_C$ is an isomorphism away from the nodes of $Z_C$.  A ``small''  resolution means that additionally, the fiber of $g^*$ over each node is of dimension 1, which for ordinary double points implies that the fibers must be copies of $\mathbb{P}^1$.  The fact that $g^*$ must be a small resolution follows from general theory and the fact that $Z_{BB}$ is Calabi-Yau.  

\section{Mirror family to hypersurfaces in $\mathbb{P}^{(1,1,2,2,2)}$}

A similar procedure of using toric morphisms to resolve a degenerate mirror family can be used for the mirror to another Calabi-Yau family, the family of quartic hypersurfaces in the weighted projective space $\mathbb{P}^{(1,1,2,2,2)}$.  This weighted projective space has a fan consisting of cones over faces of the reflexive polytope $\Delta_{WP} \subseteq M_\mathbb{R}$ with vertices $(-1,-2,-2,-2)$ and $f_1$, $f_2$, $f_3$, $f_4$ (the standard basis).  The vertices of the dual polytope $\Delta^*_{WP} \subseteq N_\mathbb{R}$ are the rows of 
\[ \left( \begin{array}{cccc}
-1 & -1 & -1 & -1 \\
7 & -1 & -1 & -1 \\
-1 & 3 & -1 & -1 \\
-1 & -1 & 3 & -1 \\ 
-1 & -1 & -1 & 3 \\
\end{array} \right).\]
Just like the toric variety $P(2,4)$, $\mathbb{P}^{(1,1,2,2,2)}$ can be embedded in $\mathbb{P}^5$ as a quadratic hypersurface, for instance, via the equation $z_0z_1 = z^2_2$ where $z_0$, $z_1$, $\dots$, $z_5$ are homogeneous coordinates on $\mathbb{P}^5$.  The   
singular locus is the plane of $A_1$ singularities where $z_0 = z_1 = z_2 = 0$.  Intersecting with a generic quartic hypersurface in $\mathbb{P}^5$ gives a variety with singular locus consisting of four lines of $A_1$ singularities. After a crepant toric resolution of $\mathbb{P}^{(1,1,2,2,2)}$, this variety is a smooth Calabi-Yau threefold.  We refer to this family of Calabi-Yau threefolds as $\mathcal{X}_{WP}$. 

By Batyrev's construction, the mirror family $\mathcal{X}^*_{WP}$ is given by (a Calabi-Yau compactification of)
$$-1+X_1+X_2 + X_3+X_4 + a_5 X^{-1}_1 (X_2 X_3 X_4)^{-2}+a_6 (X_2 X_3 X_4)^{-1} = 0$$
in $T= \hbox{Spec} \ \mathbb{C}[M]$, where $X_i = z^{f_i}$.  There is a degenerate subfamily satisfying $4a_5 = a^2_6$, which we will refer to as $\mathcal{X}^*_S$.  After factoring, the defining equation becomes:
\begin{equation}
\label{factoredform}
-1+X_1(1+(a_6/2)(X_1X_2X_3X_4)^{-1})^2+X_2+X_3+X_4 = 0.
\end{equation}
This mirror family and its degenerate subfamily were discussed by Morrison in section 3.3 of \cite{morrison}, where the degenerate subfamily is given by the condition $q_2 = 4$.  Morrison also defines a birational morphism between the degenerate subfamily and $\mathcal{X}^*_{BB}$.

In \cite{fred}, we showed that the degenerate subfamily is birational to the family $\mathcal{X}^*_{BB}$ via a toric morphism, in an entirely similar fashion to the $P(2,4)$ case.  In this case, it is more convenient to use a different complex structure parameter $b_6$ on $\mathcal{X}^*_{BB}$, so that the defining equations for $\mathcal{X}^*_{BB}$ are:
\begin{align*}
-1 + Y_1 + Y_2 + Y_3 + Y_4 &= 0 \\
-1+Y_5 + b_6 (Y_1Y_2Y_3Y_4Y_5)^{-1} &= 0
\end{align*}
To match the complex structure parameters in the two families we must set $b_6 = a_6/2$.  (Thus, from now on we will also use the parameter $b_6$ for the family $\mathcal{X}^*_S$, with the understanding that $a_6/2$ is replaced with $b_6$ in Equation \ref{factoredform}.)  The morphism is given by the linear map $h^* : N'_\mathbb{R} \rightarrow N_\mathbb{R}$ with matrix 
\[ \left( \begin{array}{ccccc}

0 & 0 & 0 & 1 & 2 \\
0 & 0 & 1 & 0 & 0 \\
0 & 1 & 0 & 0 & 0 \\
1 & 0 & 0 & 0 & 0 
\end{array} \right) \]
in the standard bases.

\subsection{Existence of the toric morphism}

For the $P(2,4)$ case, \cite{fred} also showed that the toric morphism could be extended to a regular morphism from the entire Calabi-Yau family $\mathcal{X}^*_{BB}$, rather than just a dense open subset.  However, this was not carried out fully for the weighted projective space case.  To discuss resolution of singularities of the degenerate family $\mathcal{X}^*_S$, we need to extend the toric morphism, because the singularities of $\mathcal{X}^*_S$ are not contained in the open torus.  

In the $P(2,4)$ case, the behavior of the toric morphism and polytopes on the mirror side was very nice.  In particular, we had that $g^* : N'_\mathbb{R} \rightarrow N_\mathbb{R}$ satisfies $g^*(\nabla) = \Delta^*_{P(2,4)}$.  After choosing appropriate MPCP resolutions $\widehat{X}(\nabla)$ and $\widehat{X}(\Delta^*_{P(2,4)})$ of $X(\nabla)$ and $X(\Delta^*_{P(2,4)})$, the toric morphism $g^* : \widehat{X}(\nabla) \rightarrow \widehat{X}(\Delta^*_{P(2,4)})$ exists, meaning each cone $C \in \widehat{\Sigma}(\nabla)$ is such that $g^*(C) \subseteq C'$ for some $C' \in \widehat{\Sigma}(\Delta^*_{P(2,4)})$.

As discussed in Section 6 of \cite{fred}, the picture for the $\mathbb{P}^{(1,1,2,2,2)}$ case is not quite as nice.  The linear map $h^* : N'_\mathbb{R} \rightarrow N_\mathbb{R}$ does not satisfy $h^*(\nabla) = \Delta^*_{WP}$.  Instead we have that $\Delta^*_{WP}$ is a proper subset of $h^*(\nabla)$.  Also, the primitive elements of $N'$ lying in the kernel of $h^*$ are $(0,0,0,2,-1)$ and $(0,0,0,-2,1)$.  Because $(0,0,0,-2,1)\not \in \nabla$, we cannot hope to have a toric morphism $h^*: \widehat{X}(\nabla) \rightarrow \widehat{X}(\Delta^*_{WP})$ between MPCP resolutions of $X(\nabla)$ and $X(\Delta^*_{WP})$.  (It is worth noting that we have some amount of freedom in choosing the nef partition of $\mathbb{P}^5$, which amounts to a choice of properly normalized piecewise linear functions $\varphi_1$ and $\varphi_2$ on the fan for $\mathbb{P}^5$, representing the line bundles $\mathcal{O}_{\mathbb{P}^5}(4)$ and $\mathcal{O}_{\mathbb{P}^5}(2)$ respectively.  Different piecewise linear functions will affect the orientation of the polytope $\nabla = Conv(Newt(\varphi_1),Newt(\varphi_2))$ in $N'_\mathbb{R}$, but it is possible to show that the same problems with $h^*$ arise regardless of the choice of the $\varphi_i$.)

The construction of a toric morphism that resolves $\mathcal{X}^*_{S}$ to a nonsingular Batyrev-Borisov mirror $\mathcal{X}^*_{BB}$ requires a bit more analysis because of these problems, but a similar strategy still works.  We will show that it is possible to delete some of the cones of $\Sigma(\nabla)$, obtaining a subfan $\Sigma'(\nabla)$ such that $X(\Sigma'(\nabla))$ still contains all generic members of the unresolved family in $X(\nabla)$ defined by $s_1 = s_2 = 0$, where $s_i$ is a generic global section of $\mathcal{L}_i$.  (This the family that can be resolved to $\mathcal{X}^*_{BB}$ by choosing an MPCP resolution of $X(\nabla)$.)  Then, we will show that there is a crepant toric resolution of $X(\Sigma'(\nabla))$ to a projective toric variety $\widehat{X}(\Sigma'(\nabla))$, so that $\mathcal{X}^*_{BB} \subseteq \widehat{X}(\Sigma'(\nabla))$ is a family of smooth compact CY varieties, and there is an MPCP resolution $\widehat{X}(\Delta^*_{WP})$ such that the toric morphism $h^* : \widehat{X}(\Sigma'(\nabla)) \rightarrow \widehat{X}(\Delta^*_{WP})$ exists.  This morphism acts as a resolution of singularities of $\mathcal{X}^*_{S}$.

The fan $\Sigma'(\nabla)$ is defined as follows.  
\begin{definition} \label{sigmaprime} Consider the set of faces $f$ of $\nabla$ satisfying all of the following conditions:

1. $f$ does not contain $(0,0,0,0,-1)$.

2. $f$ does not contain $(0,0,0,0,1)$.

3. $f$ does not contain the line segment \[Conv((-1,-1,-1,-1,0),(-1,-1,-1,-1,4)).\]
Define $\Sigma'(\nabla)$ as the subfan of $\Sigma(\nabla)$ consisting of cones over all such faces.  
\end{definition}

Because $\Sigma'(\nabla)$ is a subfan of $\Sigma(\nabla)$, $X(\Sigma'(\nabla))$ is naturally included into $X(\nabla)$ as an open set.

\begin{prop} Let $V \subseteq X(\nabla)$ be a subvariety of $X(\nabla)$ defined by $s_1 = s_2 = 0$ where $s_i$ is a generic global section of $\mathcal{L}_i$.  Then $V \subseteq X(\Sigma'(\nabla))$.  \end{prop}

\begin{proof} Let $\Delta_i$ be the Newton polytope of the function $\varphi^i$ which corresponds to the line bundle $\mathcal{L}_i$ on $X(\nabla)$.  Then monomial global sections of $\mathcal{L}_i$ correspond to lattice points in $\Delta_i$.  If $i=1$ then we calculate that the lattice points in $\Delta_1$ are $e_1, e_2, e_3, e_4$ and $0$, while the lattice points in $\Delta_2$ are $e_5, -e_1-e_2-e_3-e_4-e_5$ and $0$.  (Here $e_1, \dots, e_5$ is the standard basis of $M'_\mathbb{R} \cong \mathbb{R}^5$.) 

Suppose that $m \in \Delta_i$ is such that the monomial global section $z^m \in \Gamma(X(\nabla),\mathcal{L}_i)$ is nowhere vanishing on some torus orbit $T \subseteq X(\nabla)$, but all other lattice points $m' \in \Delta_i$ are such that $z^{m'}$ is identically zero on $T$.  Then it is clear that a generic section $s \in \Gamma(X(\nabla),\mathcal{L}_i)$ will not vanish anywhere on $T$ and thus $V$ will not contain any points in $T$.

Straightforward computation verifies that if $f$ is a face such that the cone over $f$ is not in $\Sigma'(\nabla)$ (i.e. $f$ does not satisfy all of the conditions in Definition \ref{sigmaprime}) then the associated torus orbit $T(f)$ will have this property for the monomial global sections of $\mathcal{L}_2$.  In particular, suppose $f$ contains $(0,0,0,0,1)$ and let $T(f)$ be the torus orbit corresponding to $f$.  Then one can check that $z^{-e_1-e_2-e_3-e_4-e_5}$, as a global section of $\mathcal{L}_2$, is nowhere vanishing on $T(f)$, but the other global sections $z^{0}$ and $z^{e_5}$ vanish everywhere on $T(f)$.  Similarly, if $f$ contains $(0,0,0,0,-1)$, then the global section $z^{e_5}$ of $\mathcal{L}_2$ is nowhere vanishing on $T(f)$ while $z^{-e_1-e_2-e_3-e_4-e_5}$ and $z^0$ vanish everywhere on $T(f)$.  If $f$ contains $Conv((-1,-1,-1,-1,0),(-1,-1,-1,-1,4))$, then $z^0$ is nowhere vanishing on $T(f)$ while $z^{-e_1-e_2-e_3-e_4-e_5}$ and $z^{e_5}$ vanish everywhere on $T(f)$.
\end{proof}

The next step in the construction is to show that there are MPCP resolutions of $X(\Delta^*_{WP})$ which are well-behaved, in a sense to be defined, with respect to the fan $\Sigma'(\nabla)$ and the map $h^*$.  As an intermediate step, we will define a fan $\Sigma'(\Delta^*_{WP})$, which is a certain partial crepant subdivision of $\Sigma(\Delta^*_{WP})$.  Any refinement of $\Sigma'(\Delta^*_{WP})$ to an MPCP subdivision will have the needed properties.  (A partial crepant projective subdivision of $\Sigma(\Delta^*_{WP})$, by definition, is subdivision of $\Sigma(\Delta^*_{WP})$ which is projective and all of whose rays are rays over lattice points in $\Delta^*_{WP}$.)

\begin{prop} \label{goodsubdiv} There exists a partial crepant projective subdivision $\Sigma'(\Delta^*_{WP})$ of $\Sigma(\Delta^*_{WP})$ such that for every cone $C \in \Sigma'(\nabla)$, $h^*(C)$ is a union of cones in $\Sigma'(\Delta^*_{WP})$. \end{prop}

\begin{proof} Data for the fan $\Sigma'(\Delta^*_{WP})$ is given in Table \ref{sigma1}.  Note that in addition to the vertices of $\Delta^*_{WP}$, $\Sigma'(\Delta^*_{WP})$ contains rays over the lattice points $(-2,0,0,2)$, $(-2,0,2,0)$, $(-2,2,0,0)$, and $(3,-1,-1,-1)$.  These are respectively the images of $(2,0,0,0,-1)$, $(0,2,0,0,-1)$, $(0,0,2,0,-1)$, and $(-1,-1,-1,3,0)$, which are vertices of $\nabla$, under $h^*$.

$\Sigma'(\Delta^*_{WP})$ was constructed by starting with the piecewise linear support function $\varphi_{\Delta^*_{WP}} : N'_\mathbb{R} \rightarrow \mathbb{R}$ which is identically equal to 1 on $\partial \Delta^*_{WP}$.  We then subtract small positive values from $\varphi_{\Delta^*_{WP}}$ at the points $(-2,0,0,2)$, $\dots$, $(3,-1,-1,-1)$ to obtain a new convex piecewise linear function $\varphi'$.  This function is strictly convex on the fan $\Sigma'(\Delta^*_{WP})$.

In order to check that $\Sigma'(\Delta^*_{WP})$ has the claimed property, we use the same approach as in \cite{fred}.  Using a script written for the computer algebra program Macaulay2 \cite{M2} and its Polyhedra package \cite{birkner}, it is possible to calculate the image of each cone in $\Sigma'(\nabla)$ under $h^*$ and show that it is a union of cones in $\Sigma'(\Delta^*_{WP})$.  The Macaulay2 code is available at \url{http://math.ucr.edu/~karl/M2code.html}.
\end{proof} 

\begin{table}
\begin{tabular}{ |l |m{9cm} | }
  \hline
 \pbox{20cm}{Generators of rays \\ of $\Sigma'(\Delta^*_{WP})$} & {\begin{gather*} v_1 = (-1,-1,-1,-1), \ v_2 = (7,-1,-1,-1), \\ 
 v_3 = (3,-1,-1,-1), \ v_4 = (-1,3,-1,-1), \\
v_5 = (-1,-1,3,-1), \ v_6 = (-1,-1,-1,3), \\ 
v_7 = (-1,1,0,0), \ v_8 = (-1,0,1,0), \\
v_9 = (-1,0,0,1) \end{gather*}} \\
  \hline
  \pbox{20cm}{Generators of maximal cones \\ of $\Sigma'(\Delta^*_{WP})$} & 
   \begin{gather*} \{ v_1, v_7, v_8, v_9 \}, \ \{ v_2,v_7,v_8,v_9 \}, \\
\{ v_1,v_3,v_5,v_6 \}, \ \{ v_1,v_3,v_4,v_6 \}, \\
\{ v_1,v_3,v_4,v_5 \}, \ \{ v_2,v_3,v_4,v_5 \}, \\
\{ v_2,v_3,v_4,v_6 \}, \ \{ v_2,v_3,v_5,v_6 \}, \\
\{ v_1,v_5,v_6,v_8,v_9 \}, \ \{ v_1,v_4,v_6,v_7,v_9 \}, \\
\{ v_1,v_4,v_5,v_7,v_8 \}, \ \{ v_2,v_5,v_6,v_8,v_9 \}, \\
\{ v_2,v_4,v_6,v_7,v_9 \}, \ \{ v_2,v_4,v_5,v_7,v_8 \}
   \end{gather*} \\
  \hline
\end{tabular}
\caption{Data for fan $\Sigma'(\Delta^*_{WP})$ in $N_\mathbb{R}$}
\label{sigma1}
\end{table}

\begin{lemma} \label{lemma:crepant} Let $\ell \in \Delta^*_{WP}$ be a nonzero lattice point, and let $r = \mathbb{R}_{\geq 0} \ell$ be the ray over $\ell$.  Then all the rays of the intersection fan $(h^*)^{-1}(r) \cap \Sigma'(\nabla)$ are rays over lattice points in $\nabla$.
\end{lemma}

\begin{proof} This can again be done by direct computation using Macaulay2's Polyhedra package.  Recalling that every cone in $\Sigma'(\Delta)$ is the cone over $f$ for some face $f \in \nabla$, it suffices to show that the intersection of the two-dimensional half space $(h^*)^{-1}(r)$ with any such $f$ is either empty, a lattice point of $\nabla$, or a line segment whose vertices are both lattice points of $\nabla$.  The Macaulay2 code is available at \url{http://math.ucr.edu/~karl/M2code.html}.
\end{proof}

\begin{prop} \label{prop:res} Let $\widehat{\Sigma}(\Delta^*_{WP})$ be the fan for an MPCP resolution $\widehat{X}(\Delta^*_{WP})$ of $X(\Delta^*_{WP})$.  Suppose that $\widehat{\Sigma}(\Delta^*_{WP})$ is also a subdivision of the fan $\Sigma'(\Delta^*_{WP})$ from Proposition \ref{goodsubdiv}.  Then there is a partial crepant  resolution $\widehat{X}(\Sigma'(\nabla))$ of $X(\Sigma'(\nabla))$, such that $\mathcal{X}^*_{BB} \subseteq \widehat{X}(\Sigma'(\nabla))$ is a family of smooth projective CY varieties, and the toric morphism $h^*: \widehat{X}(\Sigma'(\nabla)) \rightarrow \widehat{X}(\Delta^*_{WP})$ exists.  The morphism $h^*$ acts as a crepant resolution of members of the family $\mathcal{X}^*_S \subseteq \widehat{X}(\Delta^*_{WP})$.
\end{prop} 

\begin{proof} Consider the intersection fan $\Sigma_{int}$ which consists of all cones of the form $(h^*)^{-1}(C_1) \cap C_2$, where $C_1$ is a cone in $\widehat{\Sigma}(\Delta^*_{WP})$ and $C_2$ is a cone in $\Sigma'(\nabla)$.  By Proposition \ref{goodsubdiv}, $h^*(C_2)$ can be written as a union of cones in $\Sigma'(\Delta^*_{WP})$.  Since $\widehat{\Sigma}(\Delta^*_{WP})$ is a subdivision of $\Sigma'(\Delta^*_{WP})$, the argument in Proposition 5.6 of \cite{fred} shows that all rays of $(h^*)^{-1}(C_1) \cap C_2$ must be of the form $(h^*)^{-1}(\mathbb{R}_{\geq 0} \ell) \cap C'$ where $\ell$ is a lattice point of $\Delta^*_{WP}$ and $C'$ is some cone in $\Sigma'(\nabla)$.  But by Lemma \ref{lemma:crepant}, this must be the ray over a lattice point in $\nabla$.  

It follows that $\Sigma_{int}$ is the fan for a crepant partial resolution of $\Sigma'(\nabla)$.  The morphism of fans $h^* : \Sigma_{int} \rightarrow \widehat{\Sigma}(\Delta^*_{WP})$ exists by construction, and $\Sigma_{int}$ will be quasiprojective since $\widehat{\Sigma}(\Delta^*_{WP})$ is projective and $\Sigma'(\nabla)$ is quasiprojective. If $\Sigma_{int}$ is not maximally subdivided (i.e., there are non-simplicial cones in $\Sigma_{int}$, or there are nonzero lattice points $\ell \in \nabla$ which are in the support of $\Sigma_{int}$ but $\mathbb{R}_{\geq 0} \ell$ is not a ray of $\Sigma_{int}$) then we can make $\Sigma_{int}$ maximal by repeatedly taking star subdivisions at the appropriate lattice points.  This produces the fan for the needed toric variety $\widehat{X}(\Sigma'(\nabla))$.  Since it comes from a maximal crepant subdivision of $\Sigma'(\nabla)$, standard results guarantee that $\mathcal{X}^*_{BB} \subseteq \widehat{X}(\Sigma'(\nabla))$ will be a family of smooth projective CY threefolds.  

Restricting the toric morphism $h^* : \widehat{X}(\Sigma'(\nabla)) \rightarrow \widehat{X}(\Delta^*_{WP})$  to $\mathcal{X}^*_{BB}$, we get a regular birational morphism $h^*: \mathcal{X}^*_{BB} \rightarrow \mathcal{X}^*_S$.  Because it is a morphism from a family of smooth projective CY threefolds, Proposition 5.8 from \cite{fred} guarantees this must act as a resolution of singularities of members of $\mathcal{X}^*_S$. \end{proof}

\subsection{Singular locus in $\mathcal{X}^*_S$}

For the rest of this section we will analyze how the toric morphism resolves the singularities of a generic member $Z_S$ of $\mathcal{X}^*_S$.  In the proof of Proposition \ref{prop:res} we defined an intersection fan $\Sigma_{int}$ such that the morphism $h^* : X(\Sigma_{int}) \rightarrow \widehat{X}(\Delta^*_{WP})$ exists.  A partially resolved family $\mathcal{Y}^*$, birational to $\mathcal{X}^*_S$, exists in $X(\Sigma_{int})$.  We will show that $h^*$ restricts to an isomorphism between a generic member $Y \in \mathcal{Y}^*$ and its corresponding member $Z_S \in \mathcal{X}^*_S$.  Thus, the family $\mathcal{X}^*_S$ can actually be embedded into $X(\Sigma_{int})$.  As stated in Proposition \ref{prop:res}, further resolving $X(\Sigma_{int})$ to $\widehat{X}(\Sigma'(\nabla))$ will resolve $\mathcal{Y}^*$ to the family $\mathcal{X}^*_{BB}$.

Verifying that $h^*$ acts as an isomorphism between the families $\mathcal{Y}^*$ and $\mathcal{X}^*_S$ requires, at least with our approach, a significant amount of computation in local coordinates.  We will give the details for the most important calculations, which should at least make this claim seem fairly plausible.  The remaining details that need to be checked are all simpler versions of the main calculations in the text.

First we describe the singular locus of generic members $Z_S \in \mathcal{X}^*_S$.

\begin{prop} For generic values of $b_6$, the subvariety of $X(\Delta^*_{WP})$ defined by the anticanonical line bundle section $$-1+X_1(1+b_6 (X_1X_2X_3X_4)^{-1})^2+X_2+X_3+X_4=0$$ intersects all torus orbits of $X(\Delta^*_{WP})$ transversally (meaning the intersection scheme is either empty, or nonsingular and of codimension one in the torus orbit) except possibly those corresponding to the cones over the face $$F_1 = Conv((-1,3,-1,-1),(-1,-1,3,-1),(-1,-1,-1,3))$$ and all of its sub-faces. \end{prop}

\begin{proof} Excluding the open torus and torus orbits corresponding to cones over faces of $F_1$,
a torus orbit of $X(\Delta^*_{WP})$ must correspond to a cone over a face of $\Delta^*_{WP}$ which includes one of the vertices $(-1,-1,-1,-1)$ or $(7,-1,-1,-1)$.  Because the global section $(X_2 X_3 X_4)^{-1}$ of the line bundle vanishes on both of the toric divisors corresponding to the rays over these vertices, the intersection scheme with such a torus orbit will be the same as the intersection with the subvariety defined by the equation
$$-1+X_1+a_5 (X_2 X_3 X_4)^{-1}+b^2_6 X^{-1}_1 (X_2X_3X_4)^{-2}+X_2+X_3+X_4=0$$
for any value of $a_5$, as can be seen by expanding out the above equation.  Then by picking a general value of $a_5$, we can make the subvariety defined by this equation isomorphic (via a toric automorphism of $X(\Delta^*_{WP})$) to the subvariety defined by a generic section of the line bundle.  The result then follows from Proposition 3.1.3 of \cite{Ba}.
\end{proof}

By \cite{Ba}, Corollary 3.1.7, it now follows that after an MPCP resolution of $X(\Delta^*_{WP})$ to a new toric variety $\widehat{X}(\Delta^*_{WP})$, all singularities of $Z_S$ must be contained in the affine charts corresponding to cones contained in the cone over $F_1$.

\begin{figure}[t] 
  \begin{center} \includegraphics[scale=0.5]{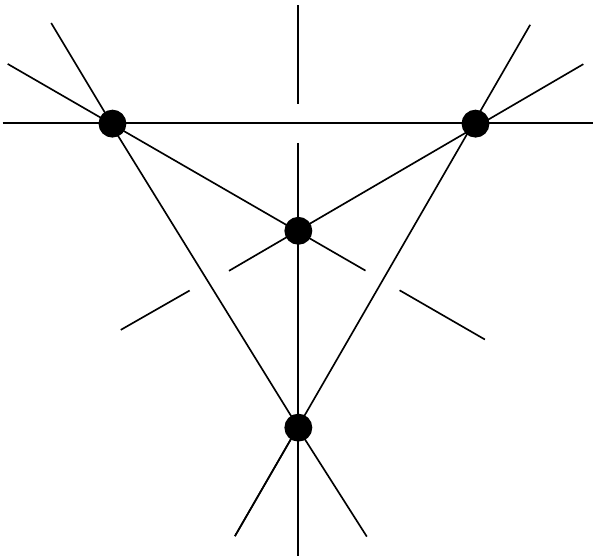} \end{center}
   \caption{\label{singularities} Configuration of singular locus in generic members of the family $\mathcal{X}^*_S$.  Each line represents a $\mathbb{P}^1$ of $A_1$ singularities, and the dots represent the intersection points $P_1, \dots, P_4$ described in the text. } 
\end{figure}

According to Proposition \ref{prop:res}, to be able to construct a toric morphism $h^* : \widehat{X}(\Sigma'(\nabla)) \rightarrow \widehat{X}(\Delta^*_{WP})$, the resolution $\widehat{X}(\Delta^*_{WP})$ must come from a refinement of $\Sigma'(\Delta^*_{WP})$.  Intersecting the cones of $\Sigma'(\Delta^*_{WP})$ with the face $F_1$, we get the configuration shown in Figure \ref{subdivisionwp1}.  Thus we can choose a resolution $\widehat{X}(\Delta^*_{WP})$ which refines $F_1$ as shown in Figure \ref{subdivchoice}.  

\begin{figure}[t] 
  \begin{center} \includegraphics[scale=0.5]{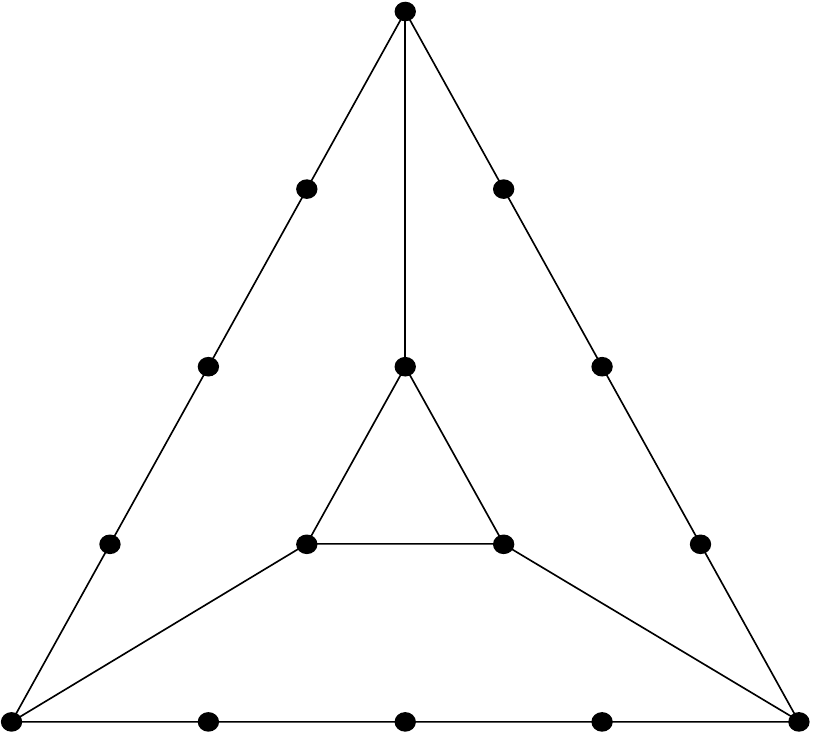} \end{center}
   \caption{\label{subdivisionwp1} Subdivision of the face $F_1 \subseteq \Delta^*_{WP}$ in the fan $\Sigma'(\Delta^*_{WP})$. } 
\end{figure}

\begin{figure}[t] 
  \begin{center} \includegraphics[scale=0.5]{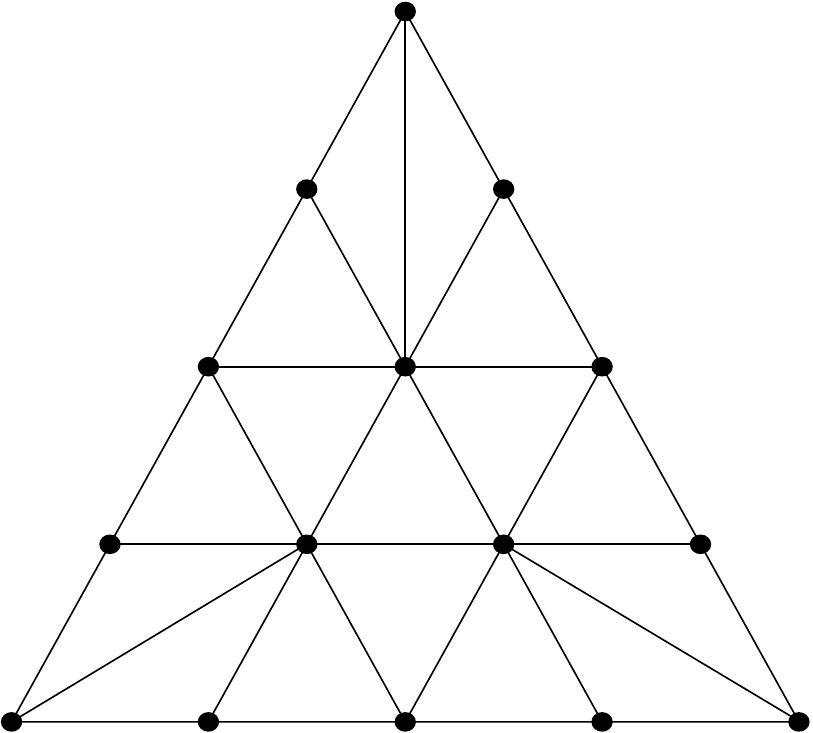} \end{center}
   \caption{\label{subdivchoice} Choice of maximal lattice subdivision of the face $F_1$. } 
\end{figure}

Examining the affine chart $W_1 \subseteq X'(\Delta^*_{WP})$, which corresponds to the central triangle contained in $F_1$ (see Table \ref{tab:y1}), we see that $Z_S \cap W_1$ is defined by the equation  
\[
-D_1D_2D_3D^{-1}_4+(1+b_6 D^{-1}_4)^2+D^2_1D_2D_3D^{-1}_4+   D_1D^2_2D_3D^{-1}_4+D_1D_2D^2_3D^{-1}_4 = 0
\]
or after rearranging and factoring,
\[
(1+b_6 D^{-1}_4)^2 = D_1D_2D_3D^{-1}_4(1-D_1-D_2-D_3) \]
This subvariety will be singular whenever we have that both $1+b_6 D^{-1}_4 = 0$ (or $D_4 = -b_6$) and any two of $D_1$, $D_2$, $D_3$, and $1-D_1-D_2-D_3$ are zero.  Thus, we get six lines of $A_1$ singularities, each corresponding to a choice of two of the above linear polynomials.  There are four points $P_1, \dots, P_4$ where three of the lines intersect, corresponding to setting three of the linear polynomials equal to zero (see Figure \ref{singularities}).  These four points are toric singularities of the type defined by the equation $x^2=yzw$ in $\mathbb{C}^4$ at $x=y=z=w=0$.  A somewhat tedious, but straightforward, analysis shows that these six $\mathbb{P}^1$s constitute the entire singular locus of $Z_S$.  

\begin{table}
\begin{tabular}{ |l |m{9cm} | }
  \hline
  Affine Chart & $$W_1$$ \\
  \hline
  Convex Cone & {\begin{gather*} \mathbb{R}_{\geq 0} (-1,1,0,0) + \mathbb{R}_{\geq 0}(-1,0,1,0)+ \\ \mathbb{R}_{\geq 0}(-1,0,0,1) \subseteq N_\mathbb{R} \end{gather*}} \\
  \hline
  Coordinate Ring & $$R_1 = \mathbb{C}[D_1, D_2, D_3, D^{\pm 1}_4]$$ {\begin{align*} D_1 &= z^{(0,1,0,0)} \\
 D_2 &= z^{(0,0,1,0)} \\ 
 D_3 &= z^{(0,0,0,1)} \\
 D_4 &= z^{(1,1,1,1)}
 \end{align*}} \\
  \hline
  Ideal of $Z_{S} \cap W_1$  & 
  {\begin{gather*} I_1 = (-D_1D_2D_3D^{-1}_4+(1+b_6 D^{-1}_4)^2+ \\ D^2_1D_2D_3D^{-1}_4+   D_1D^2_2D_3D^{-1}_4+D_1D_2D^2_3D^{-1}_4) \subseteq R_1 \end{gather*}} \\
  \hline
\end{tabular}
\caption{Data for affine chart $W_1$}
\label{tab:y1}
\end{table}

\begin{table}
\begin{tabular}{ |l |m{9cm} | }
  \hline
  Affine Chart & $$U_1$$ \\
  \hline
  Convex Cone & {\begin{gather*} \mathbb{R}_{\geq 0} (2,0,0,0,-1) + \mathbb{R}_{\geq 0}(0,2,0,0,-1)+ \\ \mathbb{R}_{\geq 0}(0,0,2,0,-1) \subseteq N'_\mathbb{R} \end{gather*}} \\
  \hline
  Coordinate Ring & $$R_2 = \mathbb{C}[B_1, B_2, B_3, B_4, B^{\pm 1}_5, B^{\pm 1}_6]/(B_1B_2B_3-B^2_4)$$ {\begin{align*} B_1 &= z^{(-1,-1,0,0,-2)} \\
 B_2 &= z^{(-1.0,-1,0,-2)} \\ 
 B_3 &= z^{(0,-1,-1,0,-2)} \\
 B_4 &= z^{(-1,-1,-1,0,-3)} \\
 B_5 &= z^{(0,0,0,1,0)} \\
 B_6 &= z^{(-1,-1,-1,0,-2)}
 \end{align*}} \\
  \hline
  Ideal of $Y \cap U_1$  & 
  {\begin{gather*} I_2 = (1+b_6B^{-1}_5B_6-B_4B^{-1}_6, \\ B_1B^{-1}_6+B_2B^{-1}_6+B_3B^{-1}_6+B_5-1) \subseteq R_2 \end{gather*}} \\
  \hline
\end{tabular}
\caption{Data for affine chart $U_1$}
\label{tab:z1}
\end{table}

\begin{table}
\begin{tabular}{ |l |m{9cm} | }
  \hline
  Affine Chart & $$U_2$$ \\
  \hline
  Convex Cone & {\begin{gather*} \mathbb{R}_{\geq 0} (2,0,0,0,-1) + \mathbb{R}_{\geq 0}(0,2,0,0,-1)+ \\ \mathbb{R}_{\geq 0}(0,0,0,2,-1) \subseteq N'_\mathbb{R} \end{gather*}} \\
  \hline
  Coordinate Ring & $$R_3 = \mathbb{C}[C_1, C_2, C_3, C_4, C^{\pm 1}_5, C^{\pm 1}_6]/(C_1C_2C_3-C^2_4)$$ {\begin{align*} C_1 &= z^{(-1,-1,0,0,-2)} \\
 C_2 &= z^{(-1.0,0,-1,-2)} \\ 
 C_3 &= z^{(0,-1,0,-1,-2)} \\
 C_4 &= z^{(-1,-1,0,-1,-3)} \\
 C_5 &= z^{(0,0,1,0,0)} \\
 C_6 &= z^{(-1,-1,0,-1,-2)}
 \end{align*}} \\
  \hline
  Ideal of $Y \cap U_2$  & 
  {\begin{gather*} I_3 = (1+b_6C^{-1}_5C_6-C_4C^{-1}_6, \\ C_1C^{-1}_6+C_2C^{-1}_6+C_3C^{-1}_6+C_5-1) \subseteq R_2 \end{gather*}} \\
  \hline
\end{tabular}
\caption{Data for affine chart $U_2$}
\label{tab:z2}
\end{table}

\subsection{Morphism from $U_1$ to $W_1$}

The morphism $h^*$ maps the affine chart $U_1 \subseteq X(\Sigma'(\nabla))$ to the affine chart $W_1 \subseteq X(\Sigma_i)$ (see Tables \ref{tab:y1} and \ref{tab:z1}), and is associated to a ring morphism $r_3: R_1 \rightarrow R_2$ defined by 
\begin{align*}
r_3(D_1) &= B_1 B^{-1}_6 \\
r_3(D_2) &= B_2 B^{-1}_6 \\
r_3(D_3) &= B_3 B^{-1}_6 \\
r_3(D_4) &= B_5 B^{-1}_6.
\end{align*}
This morphism descends to a morphism $r_3 : R_1/I_1 \rightarrow R_2/I_2$ which is an isomorphism after localizing by removing the set $1-D_1-D_2-D_3-D_4 = 0$ in $W_1$, so that we have
\[
r_3: (R_1/I_1)_{(1-D_1-D_2-D_3)} \xrightarrow{\sim} R_2/I_2. \]
The inverse is given by 
\begin{align*}
B_1 &\rightarrow D_1(1-D_1-D_2-D_3)D^{-1}_4 \\
B_2 &\rightarrow D_2(1-D_1-D_2-D_3)D^{-1}_4 \\
B_3 &\rightarrow D_3(1-D_1-D_2-D_3)D^{-1}_4 \\
B_4 &\rightarrow (1-D_1-D_2-D_3)D^{-1}_4(1+b_6 D^{-1}_4) \\
B_5 &\rightarrow 1-D_1-D_2-D_3 \\
B_6 &\rightarrow (1-D_1-D_2-D_3)D^{-1}_4.
\end{align*}
This shows that the singularity of $Z_S$ at $D_1 = D_2 = D_3 =0$ is embedded as a singularity acquired from the toric singularities of $U_1$ in $Y$.  This singular point at $D_1 = D_2 = D_3 = 0$ is one of the four points $P_i$, which we will call $P_1$.  Thus, we have proven our earlier assertion for $P_1$.  For the other three points, we need to look on different charts.

\subsection{Morphism from $U_2$ to $W_1$}
The affine chart $U_2 \subseteq X(\Sigma'(\nabla))$ is defined in Table \ref{tab:z2}.  Like $U_1$, $h^*$ maps $U_2$ to the chart $W_1 \subseteq X(\Sigma_i)$.  
This morphism is associated to a ring morphism $r_4 : R_1 \rightarrow R_3$ defined by 
\begin{align*}
r_4(D_1) &= C_5 \\
r_4(D_2) &= C_2 C^{-1}_6 \\
r_4(D_3) &= C_3 C^{-1}_6 \\
r_4(D_4) &= C_5 C^{-1}_6.
\end{align*}
As in the previous case, this map descends to a map $r_4 : R_1/I_1 \rightarrow R_3/I_3$, and becomes an isomorphism after localizing by removing $D_1 = 0$ in $W_1$, so that 
\[ r_4: (R_1/I_1)_{(D_1)} \xrightarrow{\sim} R_3/I_3. \]
The inverse is given by
\begin{align*}
C_1 &\rightarrow  D_1 D^{-1}_4 (1-D_1-D_2-D_3) \\
C_2 &\rightarrow D_1 D_2 D^{-1}_4 \\
C_3 &\rightarrow D_1 D_3 D^{-1}_4 \\
C_4 &\rightarrow D_1 D^{-1}_4(1+b_6 D^{-1}_4) \\
C_5 &\rightarrow D_1 \\
C_6 &\rightarrow D_1 D^{-1}_4.
\end{align*}
This shows, similar to the previous case, that the singularity of $Z_S$ at $D_2 = D_3 = 1-D_1-D_2-D_3 = 0$ is embedded into $Y$ as a singularity acquired from toric singularities of $U_2$. This is another of the points $P_i$, which we will call $P_2$, so we have now proven the assertion for $P_2$ and $P_1$.

For the remaining two of the $P_i$,  we can use a symmetry argument.  If $e_1$, \dots, $e_5$ is the standard basis of $N'_\mathbb{R}$, then consider any linear map $L : N'_\mathbb{R} \rightarrow N'_\mathbb{R}$ which permutes $e_1$, $e_2$, and $e_3$ and leaves $e_4$ and $e_5$ fixed.  Because $L$ leaves fixed the kernel of $h^*$, it will descend to a map $L' : N_\mathbb{R} \rightarrow N_\mathbb{R}$.  Then we have that $L' \circ h^* = h^* \circ L$.  The map $L$ will induce a toric morphism of $X(\Sigma_{int})$ which takes a member of $\mathcal{Y}^*$ to some other member of $\mathcal{Y}^*$.  With the appropriate choice of $L$, the two other points $P_3$ and $P_4$ can be mapped to singularities in $U_2$, like $P_2$, and we can use the argument for $P_2$.  

The open affine subsets of $X(\Sigma_{int})$ which contain embeddings of neighborhoods of $P_1, \dots, P_4$ correspond to the cones over the two-dimensional faces of the three-dimensional polytope
\[
\nabla'_2 = Conv((2,0,0,0,-1),(0,2,0,0,-1),(0,0,2,0,-1),(0,0,0,2,-1))
\]
(which is a face of $\nabla$).  There are four such cones, and the corresponding open affines contain one each of the points $P_1, \dots, P_4$.  Each two-dimensional face of $\nabla'_2$ is a triangle containing six lattice points.  A maximal crepant subdivision of each face will resolve the singularities of the CY family in neighborhoods of each of $P_1, \dots, P_4$.

This establishes that $h^* |_{Y}$ is an isomorphism onto a neighborhood of each point $P_1, \dots, P_4 \in Z_S$.  However, to complete the proof that $h^*$ is an isomorphism between $Y$ and $Z_S$, we still need to show that $h^*$ is an isomorphism onto the six ``points at infinity'' in the six $\mathbb{P}^1$s making up the singular locus of $Z_S$ (as shown in Figure \ref{singularities}).  

To cover the points at infinity, we must look at six other affine charts given by three-dimensional cones of $X(\Sigma_{int})$.  These are the cones over the following triangles in $N'_\mathbb{R}$:
\begin{gather*}
Conv((2,0,0,0,-1),(0,2,0,0,-1),(1,1,-1,-1,0)) \\
Conv((2,0,0,0,-1),(0,0,2,0,-1),(1,-1,1,-1,0)) \\
Conv((0,2,0,0,-1),(0,0,2,0,-1),(-1,1,1,-1,0)) \\
Conv((2,0,0,0,-1),(0,0,0,2,-1),(3,-1,-1,-1,0)) \\
Conv((0,2,0,0,-1),(0,0,0,2,-1),(-1,3,-1,-1,0)) \\
Conv((0,0,2,0,-1),(0,0,0,2,-1),(-1,-1,3,-1,0))
\end{gather*}
Each of the six charts maps onto exactly one point at infinity.  Similar to the other computations, it can be shown that $h^*|_Y$ is an isomorphism onto its image on each of these charts, establishing that $h^*$ is an isomorphism between $Y$ and $Z_S$.

\subsection{Subdividing the fan $\Sigma_{int}$}
As a final remark, we discuss the difference between the fan $\Sigma_{int}$ and the fan for the variety $\widehat{X}(\Sigma'(\nabla))$.  The only three dimensional cones in the fan $\Sigma_{int}$ which are not maximally subdivided (meaning that there are nonzero lattice points $\ell \in \nabla$ in the cone such that $\mathbb{R}_{\geq 0} \ell$ is not a ray of $\Sigma_{int}$) are exactly the cones over the two-faces of $\nabla'_2$, and the cones over the six triangles in the above list.  Members of $\mathcal{Y}^*$ are defined by generic line bundle sections.  Thus, standard results about $\Delta$-regularity from \cite{Ba, BB} say that these cones represent the only affine charts where members of the family $\mathcal{Y}^*$ could have singularities.  Furthermore, the singularities occur as transverse intersections with the singular locus of $X(\Sigma_{int})$.  So $h^*$ embeds the family $\mathcal{X}^*_S$ into $X(\Sigma_{int})$ as the family $\mathcal{Y}^*$ which has singularities acquired from the toric singularities of $X(\Sigma_{int})$.  Torically resolving $X(\Sigma_{int})$ to $\widehat{X}(\Sigma'(\nabla))$ will resolve members of $\mathcal{X}^*_S$ to smooth members of $\mathcal{X}^*_{BB}$.

\section{Doran and Harder's construction}

In \cite{dh}, Doran and Harder described a very general method for producing the mirror birational map to a toric degeneration.  This construction can be viewed either in the context of mirror symmetry between a Landau-Ginzburg model and a Fano variety, or between families of CY varieties obtained from the Batyrev-Borisov construction.  Section 4.4 of their paper discusses the application to geometric transitions of CY varieties.  As they note, their results only guarantee the existence of a birational morphism between families, not a birational contraction, which was the focus of Sections 2-3 of this paper.  However, their construction provides a general framework for understanding the birational morphisms from \cite{fred}, which previously had to be proven with a case-by-case analysis.  We also get insights in the relationship between the two singular families $\mathcal{X}^*_S$ and $\mathcal{X}^*_C$.

The construction begins with the choice of a toric variety and a nef partition.  Given the nef partition, \cite{dh} then defines an ``amenable collection of vectors subordinate to the nef partition'' (Definition 2.2), which allows a toric degeneration to be defined.

For both of the cases studied in this paper, the toric variety is $\mathbb{P}^5$. We consider $\mathbb{P}^5$ as being defined by the fan of cones over faces of the reflexive polytope $\Delta = Conv(e_1, \dots, e_5, -e_1-e_2-e_3-e_4-e_5)$ where $e_1, \dots, e_5$ is the standard basis of $M'_\mathbb{R}$.  The relevant nef partition is $\{ \varphi_1, \varphi_2 \}$, where $\varphi_1, \ \varphi_2 : M'_\mathbb{R} \rightarrow \mathbb{R}$ are integral lower convex functions which are piecewise linear on $\Sigma(\Delta)$.  The vertices of $\Delta$ on which $\varphi_1$ and $\varphi_2$ are nonzero are:
\begin{gather*}
\varphi_1(e_5) = \varphi_1(-e_1-\cdots-e_5) = 1 \\
\varphi_2(e_1) = \varphi_2(e_2) = \varphi_2(e_3) = \varphi_2(e_4) = 1
\end{gather*}
Thus, the line bundles associated to $\varphi_1$ and $\varphi_2$ are $\mathcal{O}_{\mathbb{P}^5}(2)$ and $\mathcal{O}_{\mathbb{P}^5}(4)$, respectively.  The sets $E_i$ for $i = 1, 2$ are defined as the sets of vertices $u \in \Delta$ for which $\varphi_i(u) = 1$.

In this case, an amenable collection of vectors is just a single vector $v_1 \in N'$, which must satisfy $\langle v_1, E_1 \rangle = -1$, and $\langle v_1, E_2 \rangle \geq 0$.  Notice that $0$ and $v_1$ are both elements of $Newt(\varphi_1)$, so $z^{v_1}$ and $z^0$ can be considered as global sections of the associated line bundle $\mathcal{O}_{\mathbb{P}^5}(2)$. According to Doran and Harder's construction, we then let $s_1$ degenerate to the binomial $z^{v_1}-z^0$.  

Up to isomorphism, the possible choices for $v_1$ in this case are $2f_1-f_5$ and $f_1+f_2-f_5$ where $f_1, \dots, f_5$ is the dual basis to $e_1, \dots, e_5$.  These correspond to the toric varieties in $\mathbb{P}^5$ defined in homogeneous coordinates by $z^2_0 = z_4 z_5$ and $z_0 z_1 = z_4 z_5$.  The first equation defines an embedding of weighted projective space $\mathbb{P}^{(1,1,2,2,2)}$, and the second defines an embedding of the toric variety $P(2,4)$ which is a degeneration of the Grassmannian $G(2,4)$.  $P(2,4)$ and $\mathbb{P}^{(1,1,2,2,2)}$ are both Gorenstein toric Fano varieties coming from fans $\Sigma(\Delta_{P(2,4)})$ and $\Sigma(\Delta_{WP})$, where $\Delta_{P(2,4)}$ and $\Delta_{WP}$ are reflexive polytopes.

The equations $s_1 = s_2 = 0$ in $\mathbb{P}^5$ will define a family $\mathcal{X}_{(2,4)}$ of smooth CY threefolds.  If we let $s_1$ degenerate to either of the quadratic binomials above, we get degenerations of the CY family to families of singular CY hypersurfaces in $\mathbb{P}^{(1,1,2,2,2)}$ or $P(2,4)$.  The mirrors to these degenerations, which should be birational contractions by \cite{morrison}, are what was studied in this paper.  

One of the main results of \cite{dh} is that members of the Batyrev-Borisov mirror family to $\mathcal{X}_{(2,4)}$, $\mathcal{X}^*_{BB}$, are birational to hypersurfaces in Spec $\mathbb{C}[N] \cong (\mathbb{C}^*)^4$, defined by Laurent polynomials $q_1$ and $q_2$ whose Newton polytopes are respectively $\Delta_{P(2,4)}$ and $\Delta_{WP}$.  These hypersurfaces will be open subsets of members of the degenerate families $\mathcal{X}^*_C$ and $\mathcal{X}^*_S$.  This was previously shown in \cite{fred}, but using a case-specific analysis rather than a general method.

Another result from \cite{dh} says that $q_1$ and $q_2$ are ``mutation equivalent'', meaning there is a birational morphism $\phi: (\mathbb{C}^*)^4 \rightarrow (\mathbb{C}^*)^4$ such that $\phi^*(q_1) = q_2$, and $\phi$ preserves the torus-invariant form $$\frac{dx}{x} \frac{dy}{y} \frac{dz}{z} \frac{dw}{w}$$ on $(\mathbb{C}^*)^4$ (where $x$, $y$, $z$, $w$ are coordinates on $(\mathbb{C}^*)^4$).  Using the language of \cite{acgk}, $\phi$ is an ``algebraic mutation''.  There is also an associated ``combinatorial mutation'' between the reflexive polytopes $\Delta_{P(2,4)}$ and $\Delta_{WP}$ which can be defined purely in terms of convex geometry; see \cite{acgk}, section 3.
\bibliography{myrefs}

\begin{thebibliography}{10}

\bibitem{acgk}
M.~Akhtar, T.~Coates, S.~Galkin, and A.~M. Kasprzyk.
\newblock {Minkowski Polynomials and Mutations}.
\newblock {\em {SIGMA \bf 8}}, {094}, 2012.

\bibitem{k3}
A.~C. {Avram}, M.~{Kreuzer}, M.~{Mandelberg}, and H.~{Skarke}.
\newblock {Searching for K3 fibrations}.
\newblock {\em Nuclear Physics B}, {\bf 494}:567--589, February 1997.
\newblock arXiv:hep-th/9610154.

\bibitem{Ba}
V.~V. {Batyrev}.
\newblock {Dual Polyhedra and Mirror Symmetry for Calabi-Yau Hypersurfaces in
  Toric Varieties}.
\newblock {\em {J. Alg. Geom.}}, {\bf 3}:493--535, 1994.
\newblock arXiv:alg-geom/9310003.

\bibitem{BB}
V.~V. {Batyrev} and L.~A. {Borisov}.
\newblock {On Calabi-Yau Complete Intersections in Toric Varieties}.
\newblock December 1994.
\newblock arXiv:alg-geom/9412017.

\bibitem{BCKS}
V.~V. {Batyrev}, I.~{Ciocan-Fontanine}, B.~{Kim}, and D.~{van Straten}.
\newblock {Conifold Transitions and Mirror Symmetry for Calabi-Yau Complete
  Intersections in Grassmannians}.
\newblock October 1997.
\newblock arXiv:alg-geom/9710022.

\bibitem{batkreuz}
V.~V. Batyrev and M.~Kreuzer.
\newblock {Constructing new Calabi-Yau threefolds and their mirrors via
  conifold transitions}.
\newblock {\em Adv. Theor. Math. Phys.}, {\bf 14}(3):879--898, 2010.

\bibitem{birkner}
R.~Birkner.
\newblock {{\it Polyhedra}: A package for computations with convex polyhedral
  objects}.
\newblock {\em Journal of Software for Algebra and Geometry}, {\bf
  1}(1):11--15, 2009.

\bibitem{dh}
C.~Doran and A.~Harder.
\newblock {Toric Degenerations and the Laurent polynomials related to
  Givental's Landau-Ginzburg models}.
\newblock 2015.
\newblock {arXiv:1502.02079}.

\bibitem{fred}
K.~Fredrickson.
\newblock {Mirror transitions and the Batyrev-Borisov construction}.
\newblock 2012.
\newblock {To appear in {\it Beitr. Alg. Geom.} Available at
  \url{http://math.ucr.edu/~karl/mirror.pdf}}.

\bibitem{M2}
D.~R. Grayson and M.~E. Stillman.
\newblock Macaulay2, a software system for research in algebraic geometry.
\newblock Available at http://www.math.uiuc.edu/Macaulay2/.

\bibitem{toricmorph}
Y.~{Hu}, C.-H. {Liu}, and S.-T. {Yau}.
\newblock {Toric morphisms and fibrations of toric Calabi-Yau hypersurfaces}.
\newblock {\em Adv. Theor. Math. Phys.}, {\bf 6}:457--505, 2003.
\newblock arXiv:math/0010082.

\bibitem{morrison}
D.~R. Morrison.
\newblock {Through the Looking Glass}.
\newblock In {\em Mirror Symmetry III}, pages 263--277. American Mathematical
  Society and International Press, 1999.

\bibitem{rossi}
M.~Rossi.
\newblock {Geometric Transitions}.
\newblock {\em J. Geom. Phys.}, {\bf 56}:1940--1983, 2006.
\newblock arXiv:math/0412514v1.

\end{thebibliography}
\bibliographystyle{plain}
\end{document}